\documentclass[12pt,a4paper]{article}

\usepackage{lineno,hyperref}
\modulolinenumbers[5]

\usepackage[paper=a4paper, top=3cm, bottom=3cm, left=3cm, right=3cm]{geometry} 

\usepackage{amsmath,amsthm,amsfonts}

\usepackage{bm}

\usepackage{authblk}

\allowdisplaybreaks

\theoremstyle{plain}
\newtheorem{theorem}{Theorem}[section]
\newtheorem{lemma}[theorem]{Lemma}
\newtheorem{corollary}[theorem]{Corollary}

\theoremstyle{definition}
\newtheorem{definition}[theorem]{Definition}

\theoremstyle{remark}
\newtheorem*{remark}{Remark}

\newcommand{\R}{\mathbb{R}}
\newcommand{\N}{\mathbb{N}}
\newcommand{\Z}{\mathbb{Z}}
\newcommand{\cF}{\mathcal{F}}
\newcommand{\cB}{\mathcal{B}}
\newcommand{\cD}{\mathcal{D}}
\newcommand{\cE}{\mathcal{E}}
\newcommand{\cX}{\mathcal{X}}
\newcommand{\cA}{\mathcal{A}}
\newcommand{\ep}{\varepsilon}
\newcommand{\lf}{\left\lfloor}
\newcommand{\rf}{\right\rfloor}
\newcommand{\mi}{\wedge}
\newcommand{\conv}{\leadsto}
\newcommand{\ph}{\varphi}
\newcommand{\si}{\sigma}

\begin{document}

\title{A weak convergence result for sequential empirical processes under weak dependence}

\author{Maria Mohr}
\affil{Department of Mathematics \\ University of Hamburg \\ Bundesstr. 55, 20146 Hamburg, Germany \\ maria.mohr@uni-hamburg.de}

\maketitle

\begin{abstract}
The purpose of this paper is to prove a weak convergence result for empirical processes indexed in general classes of functions and with an underlying $\alpha$-mixing triangular array of random variables. In particular, the uniformly boundedness assumption on the function class, which is required in most of the existing literature, is spared. Furthermore, under strict stationarity a weak convergence result for the sequential empirical process indexed in function classes is obtained as a direct consequence. Two examples in mathematical statistics, that cannot be treated with existing results, are given as possible applications.
\end{abstract}

\noindent\text{AMS 2010 Classification:} Primary 60F05, Secondary 60F17\\
\noindent\text{Keywords:} sequential empirical processes, functional central limit theorems, stochastic equicontinuity, $\alpha$-mixing


\section{Introduction}\label{Introduction}

The asymptotic behavior of empirical processes has been studied for decades. Inspired by the study of the empirical distribution function, more generally empirical processes indexed in function classes gained a lot of attention. In particular, central limit results, i.e.~weak convergence of the sequence of the stochastic processes to a Gaussian process, are of interest. Such results are sometimes referred to as a uniform central limit theorem (CLT) for the empirical process indexed in function classes and as a uniform functional central limit theorem (FCLT) for the partial sum process indexed in function classes, also known as the sequential empirical process indexed in function classes.

The most simple case is given if the underlying process is a family of i.i.d.~random variables. In this situation many results are available. Ossiander \cite{Ossiander1987897} showed a uniform CLT under a metric entropy condition on the function class. The uniform FCLT follows directly by the non-functional one. For example van der Vaart \& Wellner \cite{vanderVaart1996} state this result in Section 2.12 of their book and also give a great overview on empirical processes for the i.i.d.~case in general. For dependent data much less is known. There are several results concerning the non-sequential case. Doukhan, Massart \& Rio \cite{Doukhan1996393} showed a uniform CLT under a metric entropy condition on the function class and $\beta$-mixing, strictly stationary data. Dedecker \& Louhichi \cite{Dedecker2002137} generalized this result, imposing a condition on suitable maximal inequalities for the empirical process indexed in finite sets of functions. Their result is applicable to $\beta$-mixing and non-uniform $\phi$-mixing sequences. Andrews \& Pollard \cite{Andrews1994549} showed a uniform CLT for $\alpha$-mixing arrays and uniformly bounded function classes, satisfying a metric entropy condition. Massart \cite{Massart1987} showed a uniform CLT for uniformly bounded function classes and strictly stationary, $\alpha$-mixing sequences, when the mixing coefficient decays exponentially fast. Given uniformly bounded function classes, Hariz \cite{Hariz2005339} gave more general conditions in terms of bracketing numbers with respect to a norm resulting from a moment inequality satisfied by the underlying process. He particularly improves among others the results in \cite{Massart1987} and \cite{Andrews1994549}. Hansen \cite{Hansen1996347} proved a uniform CLT for mixingale arrays and classes of Lipschitz-continuous functions. More recent results use alternative dependence conditions. Hagemann \cite{Hagemann2014188} uses an alternative short-range dependence condition, applicable to non-linear time series models, and uniformly bounded classes of functions. Dehling, Durieu \& Tusche \cite{Dehling20141372} showed a uniform CLT for multiple mixing and strictly stationary data, and uniformly bounded function classes. In the dependent setup the convergence of the sequential process does not follow directly by the convergence of the non-sequential one, but requires additional conditions. Dehling, Durieu \& Tusche \cite{Dehling201487} extended their aforementioned uniform CLT to a functional version. Volgushev \& Shao \cite{Volgushev2014390} established more general assumptions, in terms of a strong version of asymptotic equicontinuity for the non-sequential process, under which a uniform FCLT holds, for strictly stationary data.

An intensive study of the literature led to two main findings. First, most uniform central limit results for dependent data impose the condition of uniformly bounded classes of functions or strong smoothness conditions. And secondly, very few results are available regarding the uniform FCLT. The aim of this paper is to prove a uniform CLT for empirical processes with an $\alpha$-mixing underlying triangular array and indexed by a function class, that satisfies a metric entropy condition. It is a generalization of the result of Andrews \& Pollard \cite{Andrews1994549} to unbounded function classes. The result particularly implies the strong version of asymptotic equicontinuity, needed in \cite{Volgushev2014390}. In the case of strict stationarity, a uniform FCLT can therefore be obtained simultaneously.

The remainder of this paper is organized as follows. In Section \ref{Uniform CLT} some definitions are recalled and the main results are displayed. Section \ref{Applications} contains two applications of the results. All proofs can be found in Section \ref{Proofs}.


\section{A uniform CLT and FCLT}\label{Uniform CLT}

In this section, the concept of weak convergence of stochastic processes in some metric space is recalled and the empirical process is defined. Furthermore, the notions of strongly mixing and bracketing numbers are presented. The main result is given in Theorem \ref{conv}. The uniform CLT and uniform FCLT are stated in Corollary \ref{cor1} and Corollary \ref{cor2} respectively. 


\subsection{Definitions and notations}

Let $T$ be an arbitrary set and let $l^{\infty}(T)$ be the set of all uniformly bounded real-valued functions on $T$. Following the modern empirical process theory, well summarized in \cite{vanderVaart1996}, this set is equipped with the supremum norm and the corresponding Borel $\sigma$-algebra. 

A stochastic process $Z=\{Z(t):t\in T\}$, defined on some underlying probability space, can be viewed as a random element in $l^{\infty}(T)$ if all sample paths are bounded.

\begin{definition}[Weak convergence]
Let $Z=\{Z(t):t\in T\}$ and $Z_n=\{Z_n(t):t\in T\}, n\in\N$, be stochastic processes in $l^{\infty}(T)$ and let $Z$ be measurable with respect to the Borel $\sigma$-algebra. Then $Z_n$ is said to \textit{converge weakly} to $Z$, denoted by $Z_n\conv Z$, if 
\[E^*[H(Z_n)]\to E[H(Z)],\]
for all bounded and continuous functions $H:l^{\infty}(T)\to\R$, where $E^*[X]$ denotes the outer expectation of a possibly non-measurable real valued mapping $X$.
\end{definition}

As it can be seen for example by applying Theorem 1.5.7 and Theorem 1.5.4 in \cite{vanderVaart1996}, it holds that $Z_n$ converges weakly to $Z$ in $l^{\infty}(T)$, if and only if the following two conditions hold
\begin{enumerate}
\item[$\bullet$] \textit{fidi convergence:} for all $K\in\N$ and all $t_1,\dots,t_K\in T$ 
\[(Z_n(t_k))_{k=1,\dots,K}\overset{\cD}{\to} (Z(t_k))_{k=1,\dots,K},\]
\item[$\bullet$] there exists a semi metric $d$ on $T$, such that $(T,d)$ is totally bounded and $Z_n$ is \textit{asymptotic equicontinuous}, i.e.
\[\lim\limits_{\delta\searrow 0}\limsup\limits_{n\to\infty}P^*\left(\sup\limits_{\{t_1,t_2\in T: d(t_1,t_2)<\delta\}}\left|Z_n(t_1)-Z_n(t_2)\right|>\epsilon\right)=0,\]
for all $\epsilon>0$, where $P^*(A)$ denotes the outer probability of a possibly non-measurable set $A$.
\end{enumerate}

These two conditions are in many situations easier to verify.

\begin{definition}[Empirical process]\label{empirical}
Let $\{X_{n,t}:1\le t\le n,n\in\N\}$ be a triangular array of random variables with values in some measure space $\cX$. For some measurable function $\ph:\cX\to\R$ and some $s\in[0,1]$, let
\[G_n(s,\ph):=\frac{1}{\sqrt{n}}\sum\limits_{i=1}^{\lf ns \rf}\left(\ph(X_{n,i})-E[\ph(X_{n,i})]\right), n\in\N.\]
For some function class $\cF$ of measurable functions $\cX\to\R$ the \textit{(non-sequential) empirical process} indexed by $\cF$ is defined as $\left\{G_n(1,\ph):\ph\in\cF\right\}, n\in\N$, and can be viewed as a sequence of random elements in $l^{\infty}(\cF)$. The \textit{sequential empirical process} indexed by $[0,1]\times\cF$ is defined as $\left\{G_n(s,\ph):s\in[0,1],\ph\in\cF\right\}, n\in\N$, and can be viewed as a sequence of random elements in $l^{\infty}([0,1]\times\cF)$.
\end{definition}

For the empirical process to converge weakly to a centered Gaussian process, assumptions on the underlying triangular array process $\{X_{n,t}:1\le t\le n,n\in\N\}$ and on the function class $\cF$ are needed. The main result will be shown for an underlying family of strongly mixing random variables and function classes, satisfying a metric entropy condition in terms of the bracketing notion. The following definition can for example be found in \cite{Su2013187}.

\begin{definition}[Strongly mixing]
For some triangular array of random variables $\{X_{n,t}:1\le t\le n,n\in\N\}$ define 
\begin{align*}
\alpha_n(t):=\begin{cases} \sup\limits_{1\le k \le n-t}\sup\limits_{\substack{A\in\sigma(X_{n,j}:1\le j\le k)\\ B\in\sigma(X_{n,j}:k+t\le j\le n)}}|P(A\cap B)-P(A)\cdot P(B)|, \ &t\le n-1 
\\ 0, &t\ge n \end{cases}
\end{align*}
and 
\begin{align*}
\alpha(t):=\begin{cases} \sup\limits_{n\in\N}\alpha_n(t), \ &t\in\N \\ 1, &t=0 \end{cases}.
\end{align*}

Then $\{X_{n,t}:1\le t\le n,n\in\N\}$ is called \textit{strongly mixing} or $\alpha$-\textit{mixing} if $\alpha(t)\to 0$ as $t\to\infty$, and $\alpha(\cdot)$ is called the mixing coefficient.
\end{definition}

For a strictly stationary sequence of random variables $\{X_t:t\in\Z\}$ the strongly mixing notion simplifies to
%
%
\begin{align*}
\alpha(t):=\sup\limits_{\substack{A\in\sigma(X_j:j\le 0)\\ B\in\sigma(X_j:j\ge t)}}|P(A\cap B)-P(A)\cdot P(B)|\to 0, t\to\infty,
\end{align*}

which can for example be found in \cite{Bradley1985165}. The following definition is based on Definition 2.1 of Andrews \& Pollard \cite{Andrews1994549}, but uses a slightly different notation.

\begin{definition}[Bracketing number]\label{brackets}
Let $\cX$ be a measure space, $\cF$ some class of functions $\cX\to\R$ and $\rho$ some semi-norm on $\cF$. For all $\ep>0$, let $N=N(\ep)$, be the smallest integer, for which there exist a class of functions $\cX\to\R$, denoted by $\cB$ and called \textit{bounding class}, and a function class $\cA \subset \cF$, called \textit{approximating class}, such that
\[|\cB|=|\cA|=N,\]
\[\rho(b)<\ep,\ \forall \ b\in\cB\]
and for all $\ph\in\cF$ there exist an $a^*\in\cA$ and a $b^*\in\cB$, such that
\[|\ph-a^*|\le b^*.\]
Then $N(\ep)$ is called the \textit{bracketing number} and denoted by $N_{[~]}(\ep,\cF,\rho)$. The function $a^*$ is referred to as the \textit{(to $\ph$) corresponding approximating function} and the function $b^*$ as the \textit{(to $\ph$) corresponding bounding function}. 
\end{definition}


\subsection{Main results}

Theorem \ref{conv} gives conditions on the underlying array process $\{X_{n,t}:1\le t\le n,n\in\N\}$ and on the function class $\cF$, under which the empirical process $\{G_n(1,\ph):\ph\in\cF\}$ satisfied a strong form of asymptotic equicontinuity.

\begin{theorem}[Equicontinuity] \label{conv}
Let $\{X_{n,t}:1\le t\le n,n\in\N\}$ be a triangular array of random variables with values in some measure space $\cX$. Let $\cF$ be a class of measurable functions $\cX\to\R$. Let furthermore the following assumptions hold.
\begin{enumerate}
\item[\textbf{(A1)}] Let $\{X_{n,t}:1\le t\le n,n\in\N\}$ be strongly mixing with mixing coefficient $\alpha(\cdot)$, such that
\begin{align*}
\sum\limits_{t=1}^{\infty}t^{Q-2}\alpha(t)^{\frac{\gamma}{2+\gamma}}<\infty,
\end{align*}
for some $\gamma>0$ and some even $Q\ge 2$. 
\item[\textbf{(A2)}] For $Q$ and $\gamma$ from assumption \textbf{(A1)} and 
\[\rho(\ph):=\sup\limits_{n\in\N}\sup_{1\le t\le n}E[|\ph(X_{n,t})|^2]^{\frac{1}{2}},\] 
for all measurable functions $\ph:\cX\to\R$, let
\[\int\limits_{0}^{1}x^{-\frac{\gamma}{2+\gamma}}\left(N_{[~]}(x,\cF,\rho)\right)^{\frac{1}{Q}}dx<\infty.\]
Furthermore, assume that each $\ep>0$ allows for a choice of bounding class $\cB$, such that for all $i=2,\dots,Q$
\begin{align}
\sup\limits_{n\in\N}\sup_{1\le t\le n}E\left[|b(X_{n,t})|^{i\frac{2+\gamma}{2}}\right]^{\frac{1}{2}}\le\ep, \ \forall \ b\in\cB. \label{eq:bounding function}
\end{align}
\end{enumerate}
Then with $d(\ph,\psi):=\sup\limits_{n\in\N}\sup\limits_{1\le t\le n}E\left[|\ph(X_{n,t})-\psi(X_{n,t})|^{Q\frac{2+\gamma}{2}}\right]^{\frac{1}{Q}\frac{2}{2+\gamma}}$, it holds that
\begin{align*}
\lim\limits_{\delta\searrow 0}\limsup\limits_{n\to\infty}E^*\left[\sup\limits_{\{\ph,\psi \in\cF:d(\ph,\psi)<\delta\}}\left|G_n(1,\ph)-G_n(1,\psi)\right|^Q\right]^{\frac{1}{Q}}=0.
\end{align*}
\end{theorem}

The proof is given in Section \ref{Proofs}. A uniform CLT as direct consequence is obtained and stated in Corollary \ref{cor1}. In contrast to most results for strongly mixing sequences in the literature, it does not require uniformly bounded function classes.

\begin{remark}
Note that assumption \textbf{(A1)} and the first part of assumption \textbf{(A2)} require a trade off between the rate of decay in the mixing coefficients and the rate of growth in the bracketing numbers. For instance if $N_{[~]}(x,\cF,\rho)=O(x^{-d})$ for $x\to 0$ and $\alpha(t)=O(t^{-\beta})$ for $t\to \infty$, for some $d,\beta>0$, then $Q$ and $\gamma$ need to be chosen such that $Q>d(\gamma/2+1)$ and $\beta>(Q-1)(2/\gamma+1)$. The assumptions are closely related to the ones made by Andrews \& Pollard \cite{Andrews1994549} with \textbf{(A1)} being slightly less restrictive than the corresponding one in \cite{Andrews1994549}.
\end{remark}

\begin{corollary}[Uniform CLT]\label{cor1}
Let the assumptions of Theorem \ref{conv} hold and let additionally for all $K\in\N$ and $\ph_k\in\cF$, $k=1,\dots,K$
\[\left(G_n(1,\ph_k)\right)_{k=1,\dots,K}\overset{\cD}{\longrightarrow} \left(G(1,\ph_k)\right)_{k=1,\dots,K},\]
where $G:=\{G(1,\ph):\ph\in\cF\}$ is a centered Gaussian process. Then
\begin{align*}
\left\{G_n(1,\ph):\ph\in\cF\right\}\conv G,
\end{align*}
in $l^{\infty}(\cF)$.
\end{corollary}

The proof is given in Section \ref{Proofs}. Note that Corollary \ref{cor1} does not require stationarity. If additionally strict stationarity is assumed, a uniform FCLT can be obtained, which is stated in the following.

\begin{corollary}[Uniform FCLT]\label{cor2}
Let $\{X_t:t\in\Z\}$ be a strictly stationary sequence of random variables. Let the assumptions of Theorem \ref{conv} be satisfied by $X_{n,t}:=X_t$, for all $1\le t\le n,n\in\N$ and by $Q> 2$. Additionally, let
\begin{enumerate}
\item[\textbf{(A3)}] $\cF$ possess an envelope function $F$, with $E[|F(X_1)|^Q]<\infty$ and let there exist a constant $M<\infty$, such that
\[\sup\limits_{\ph\in\cF}E\left[ |\ph(X_1)|^{Q\frac{2+\gamma}{2}}\right]\le M.\]
\end{enumerate}
Furthermore, let for all $K\in\N$ and all finite collections $\ph_k\in\cF$, $s_k\in[0,1]$, $k=1,\dots,K$,
\[\left(G_n(s_k,\ph_k)\right)_{k=1,\dots,K}\overset{\cD}{\longrightarrow} \left(G(s_k,\ph_k)\right)_{k=1,\dots,K},\]
where $G:=\{G(s,\ph):s\in[0,1],\ph\in\cF\}$ is a centered Gaussian process. Then
\begin{align*}
\left\{G_n(s,\ph):s\in[0,1],\ph\in\cF\right\}\conv G,
\end{align*}
in $l^{\infty}([0,1]\times \cF)$.
\end{corollary}

The proof is given in Section \ref{Proofs}. Note that due to the strict stationarity assumption, condition (\ref{eq:bounding function}) in Theorem \ref{conv} simplifies to
\[E\left[|b(X_1)|^{i\frac{2+\gamma}{2}}\right]^{\frac{1}{2}}\le \ep, \ \forall \ b\in\cB, \ \forall \ i=2,\dots,Q,\]
and the semi-norm $d$ on $\cF$ simplifies to
\[d(\ph,\psi):=E\left[|\ph(X_1)-\psi(X_1)|^{Q\frac{2+\gamma}{2}}\right]^{\frac{1}{Q}\frac{2}{2+\gamma}}.\]


\section{Applications}\label{Applications}

Empirical process theory is a powerful tool for proofs of asymptotic results in mathematical statistics. In this section, two examples will be shown that cannot be treated with results from the mentioned literature, but where Theorem \ref{conv} is applicable. Note that the result is also applied in a working paper of Mohr \& Neumeyer in the context of changepoint detection in time series regression models (see proof of Theorem 3.1. (ii) in \cite{Mohr2019}).


\subsection{A special SETAR-model}

An interesting example of a nonlinear times series model is the so called \textit{self exciting threshold autoregressive} (SETAR) model. See Tong \cite{Tong1983} for a comprehensive introduction. As a special case of this model class, consider a stochastic process $\{Y_t\in\R:t\in\Z\}$, such that
\begin{align}
Y_t=
\begin{cases}
\mu_1+\ep_t, & \text{ if } \ Y_{t-1}\le z \\
\mu_2+\ep_t, & \text{ if } \ Y_{t-1}>z
\end{cases}, \ t=1,\dots,n, \label{eq:SETAR_01}
\end{align}
with unknown $\mu_1,\mu_2\in\R$ and $z\in\R$, called the \textit{threshold}. Let $\{\ep_t\in\R:t\in\Z\}$ be strictly stationary and strongly mixing with mixing coefficients $\alpha(\cdot)$ satisfying
\[\sum\limits_{t=1}^{\infty}t^2\alpha(t)^{\frac{\gamma}{2+\gamma}}<\infty,\]
for some $\gamma\in(0,2)$. Let additionally the following moment constraints hold
\[E[\ep_t|\cF^{t-1}]=0,E\left[\ep^2_t|Y_{t-1}\right]=\si^2\in(0,\infty) \text{ and } E\left[|\ep_t|^{2(2+\gamma)}|Y_{t-1}\right]\le c\in(0,\infty),\]
a.s. where $\cF^{t-1}:=\si(\ep_j,Y_j:j\le t-1)$. 
%
Given observations $Y_0,\dots,Y_n$, it can be tested, whether there exists a $z\in\R$, such that (\ref{eq:SETAR_01}) holds with $\mu_1\neq \mu_2$, by considering
\begin{align}
\hat{\mu}_1(z)-\hat{\mu}_2(z)\label{eq:SETAR_02}
\end{align}
for $z\in\R$, where
\[\hat{\mu}_1(z):=\frac{1}{\hat{F}_Y(z)}\frac{1}{n}\sum\limits_{i=1}^{n}Y_iI\{Y_{i-1}\le z\}, \hat{\mu}_2(z):=\frac{1}{1-\hat{F}_Y(z)}\frac{1}{n}\sum\limits_{i=1}^{n}Y_iI\{Y_{i-1}> z\}\]
and $\hat{F}_Y(z):=\frac{1}{n}\sum_{i=1}^{n}I\{Y_{i-1}\le z\}$. Under
\[H_0: \ \exists \mu\in\R: \mu_1=\mu_2=\mu,\] 
the term in (\ref{eq:SETAR_02}) is expected to be small for all $z\in\R$, and large for some $z\in\R$ under the alternative $H_1: \ \mu_1\neq\mu_2$. 
Under $H_0$, it holds that uniformly in $z\in\R$
\begin{align*}
\sqrt{n}\hat{F}_Y(z)(1-\hat{F}_Y(z))\left(\hat{\mu}_1(z)-\hat{\mu}_2(z)\right)
&=\frac{1}{\sqrt{n}}\sum\limits_{i=1}^{n}\ep_i\left(I\{Y_{i-1}\le z\}-F_Y(z)\right)+o_P(1),
\end{align*}
where $F_Y$ is the distribution function of $Y_t$ for all $t\in\Z$, which does not depend on $t$ under $H_0$. 
%
Condition \textbf{(A1)} of Theorem \ref{conv} is then satisfied for $\{(\ep_t,Y_{t-1})\in\R\times \R:t\in\Z\}$ under $H_0$ by assumption. Furthermore, defining the function class
\[\cF:=\{(\ep,y)\mapsto \ep (I\{y\le z\}-F_Y(z)):z\in\R\},\]
condition \textbf{(A2)} of Theorem \ref{conv} holds for $Q=4$ and $\gamma\in(0,2)$ from above.
The existence of the $2(2+\gamma)$-th absolute moments of $\ep_t$ conditioned on $Y_{t-1}$ is needed to additionally ensure the validity of (\ref{eq:bounding function}). Hence, Theorem \ref{conv} is applicable to the empirical process
\[\left\{\frac{1}{\sqrt{n}}\sum\limits_{i=1}^{n}\left(\ph(\ep_i,Y_{i-1})-E[\ph(\ep_i,Y_{i-1})]\right):\ph\in\cF\right\}, n\in\N.\]

Note that $\cF$ is not uniformly bounded and therefore, to the best of our knowledge, non of the existing literature is applicable. Identifying
\[T_n:=\left\{T_n(z):=\frac{1}{\sqrt{n}}\sum\limits_{i=1}^{n}\ep_i\left(I\{Y_{i-1}\le z\}-F_Y(z)\right):z\in\R\right\}, n\in\N\]
with above process, it holds that $T_n$ converges weakly to a centered Gaussian process $G$ with covariance function
\[Cov(G(z_1),G(z_2))=\si^2(F_Y(z_1\mi z_2)-F_Y(z_1)F_Y(z_2)),\]
%

Let $B_0$ denote a Brownian Bridge on $[0,1]$ and $\hat{\si}^2$ be some consistent estimator for $\si^2$. If $F_Y$ is continuous, it then holds that
\begin{align*}
T_{n1}:=
\frac{1}{\sqrt{\hat{\si}^2}}\sup\limits_{z\in\R}|T_n(z)|
\overset{\cD}{\to}\frac{1}{\sqrt{\si^2}}\sup\limits_{z\in\R}|G(z)|
\overset{\cD}{=}\sup\limits_{z\in\R}|B_0(F_Y(z))|
=\sup\limits_{s\in[0,1]}|B_0(s)|
\end{align*}
and
\begin{align*}
T_{n2}:=
\frac{1}{\hat{\si}^2}\int\limits_{-\infty}^{\infty}|T_n(z)|^2w(z)dz
\overset{\cD}{\to}\frac{1}{\si^2}\int\limits_{-\infty}^{\infty}|G(z)|^2w(z)dz
\overset{\cD}{=}\int\limits_{-\infty}^{\infty}|B_0(F_Y(z))|^2w(z)dz,
\end{align*}
for some weighting function $w:\R\to\R$, such that the integral exists. Note that if $Y_t$ is absolutely continuous with density $f_Y$, then for $w\equiv f_Y$, it holds that 
\[T_{n2}\overset{\cD}{\rightarrow} \int\limits_{0}^{1}|B_0(s)|^2ds.\]

In this case, both limiting distributions are free of unknown parameters and asymptotic tests for $H_0$ based on $T_{n1}$ and $T_{n2}$ can be constructed.


\subsection{Changepoint detection in a regression model with dependent innovations}

The following example is a generalization of the nonparametric changepoint test of Burke \& Bewa \cite{Burke2013261}. Let $\{(Y_t,X_t)\in\R\times\R^d:t\in\Z\}$ be a stochastic process with
\begin{align*}
Y_t=m_t(X_t)+U_t, \ t=1,\dots,n,
\end{align*}
with unknown $m_t:\R^d\to\R$, innovations $\{U_t\in\R:t\in\Z\}$ and i.i.d.~regressors $\{X_t\in\R^d:t\in\Z\}$ with distribution function $F_X$. For some unknown $m:\R^d\to\R$, not depending on $t\in\Z$, consider the following null hypothesis
\begin{align*}
H_0: m_t(\cdot)=m(\cdot), \ \forall \ t=1,\dots,n.
\end{align*}

Furthermore, for some unknown $\sigma:\R^d\to\R$ and some strictly stationary, strongly mixing sequence $\{\ep_t\in\R:t\in\Z\}$, that is mutually independent of the process of the regressors, let
\[U_t=\sigma(X_t)\ep_t, \ t=1,\dots,n.\]

Let furthermore for some $\gamma>0$ and some even $Q>d(2+\gamma)$ the following moment constraints hold
\begin{align*}
&E\left[\ep_t|\cF^{t-1}\right]=0, \text{ with } \cF^{t-1}=\sigma(\ep_j:j\le t-1), \ \forall \ t\in\Z,\\
&E[\ep_1^2]=1, \ E\left[|\ep_1|^{Q\frac{2+\gamma}{2}}\right]<\infty,
\end{align*}
the following integrals exist
\begin{align*}
\int |m(u)|^{Q\frac{2+\gamma}{2}}dF_X(u)<\infty, \ \int |\sigma(u)|^{Q\frac{2+\gamma}{2}}dF_X(u)<\infty,
\end{align*}
and with $\alpha(\cdot)$ being the mixing coefficients, the following series converge
\[\sum\limits_{t=1}^{\infty}t^{Q-2}\alpha(t)^{\frac{\gamma}{2+\gamma}}<\infty.\]

A test for $H_0$ could be based on the following test statistic
\begin{align*}
\beta_n(s,z):=\frac{1}{\sqrt{n}}\left(\sum\limits_{i=1}^{\lf ns \rf}Y_iI\{X_i\le z\}-\frac{\lf ns \rf}{n}\sum\limits_{i=1}^{n}Y_iI\{X_i\le z\}\right),
\end{align*}
for $s\in[0,1]$ and $z\in\R^d$. Burke \& Bewa \cite{Burke2013261} consider this test in an i.i.d.~setting. It can be shown that
\begin{align*}
\beta_n(s,z)=\alpha_n(s,z)-s\alpha_n(1,z)+o_P(1),
\end{align*}
uniformly in $s\in[0,1]$ and $z\in\R^d$, where 
\[\alpha_n(s,z):=\frac{1}{\sqrt{n}}\sum\limits_{i=1}^{\lf ns \rf}\left(Y_iI\{X_i\le z\}-E[Y_iI\{X_i\le z\}]\right).\]

Under $H_0$, $\{(Y_t,X_t)\in\R\times\R^d:t\in\Z\}$ is strictly stationary and condition \textbf{(A1)} of Theorem \ref{conv} is satisfied by assumption. Defining 
\[\cF:=\left\{(y,x)\mapsto yI\{x\le z\}:z\in\R^d\right\}\] 
it holds that $N_{[~]}(\epsilon,\cF,L_2(O))=O(\epsilon^{-2d})$ where $X_1\sim P$. The integral in condition \textbf{(A2)} therefore exists for all $Q>d(2+\gamma)$. The moment assumptions on $\ep_1$, $m$ and $\sigma$ are needed to additionally ensure the validity of (\ref{eq:bounding function}) and condition \textbf{(A3)} of Corollary \ref{cor2}. An envelope function is given by $F:\R\times\R^d\to\R; (y,x)\mapsto y$. Hence, Corollary \ref{cor2} is applicable to the sequential empirical process
\[\left\{\frac{1}{\sqrt{n}}\sum\limits_{i=1}^{\lf ns \rf}\left(\ph(Y_i,X_i)-E[\ph(Y_i,X_i)]\right):s\in[0,1],\ph\in\cF\right\}, n\in\N.\]

Identifying $\left\{\alpha_n(s,z):s\in[0,1],z\in\R^d\right\}$ with above process, it can be shown that it converges weakly to a centered Gaussian process. Together with the continuous mapping theorem, this implies the weak convergence of $\left\{\beta_n(s,z):s\in[0,1],z\in\R^d\right\}$ to a centered Gaussian process $\Gamma$ with covariance function
\[Cov(\Gamma(z_1,s_2),\Gamma(z_2,s_2))=(s_1\mi s_2-s_1s_2)\left(H_{1}(z_1\mi z_2)+H_2(z_1\mi z_2)-G_1(z_1)G_1(z_2)\right),\]
where 
\begin{align*}
H_{1}(z)&:=\int_{(-\infty,z]}m^2(u)dF_X(u), \ G_1(z):=\int_{(-\infty,z]}m(u)dF_X(u),
\\
H_2(z)&:=\int_{(-\infty,z]}\sigma^2(u)dF_X(u).
\end{align*} 


\section{Proofs of main results}\label{Proofs}

The key tool in proving Theorem \ref{conv} is a moment inequality for $G_n(1,\ph)$, i.e. for the empirical process evaluated at some function $\ph$, which is stated in the following lemma. It is a generalization of Lemma 3.1 of Andrews \& Pollard \cite{Andrews1994549}, who proved a moment inequality for bounded, strongly mixing random variables. Extending this result to unbounded random variables, makes it possible to extend the uniform CLT to unbounded function classes. Nevertheless, it comes at the cost of moment assumptions. Note that similar results are available, for example Theorem 2 on page 26 in \cite{Doukhan1995} or Corollary A.0.1 on page 319 in \cite{Politis1999}.

\begin{lemma}\label{moment inequality}
Let $\{Z_{n,t}:1\le t\le n,n\in\N\}$ be a strongly mixing triangular array of random variables with values in $\R$ and with mixing coefficient $\alpha(\cdot)$. Let furthermore for some even $Q\ge 2$ and some $\gamma>0$, $\tau>0$
\begin{enumerate}
\item[\textbf{(i)}] $\sum\limits_{t=1}^{\infty}t^{Q-2}\alpha(t)^{\frac{\gamma}{2+\gamma}}<\infty$ and
\item[\textbf{(ii)}] $E[Z_{n,t}]=0$, $E\left[|Z_{n,t}|^{i\frac{2+\gamma}{2}}\right]\le \tau^{2+\gamma}$, for all $i=2,\dots,Q$ and $1\le t\le n,n\in\N$.
\end{enumerate}
Then
\begin{equation}
E\left[\left|\frac{1}{\sqrt{n}}\sum\limits_{i=1}^{n}Z_{n,i}\right|^Q\right]^{\frac{1}{Q}}\le C \max\left(n^{-\frac{1}{2}},\tau\right), \ \forall \ n\in\N,\label{eq:moment inequality}
\end{equation}
for some constant $C$ only depending on $Q, \gamma$ and the mixing coefficient.
\end{lemma}

For the proof of Lemma \ref{moment inequality}, the following covariance inequality for strongly mixing triangular arrays is used. It was stated by Sun and Chiang \cite{Su19973} (see Lemma 2.1) for $\alpha$-mixing sequences of real valued random variables. Su and Xiao \cite{Su2008347} extended it to $\alpha$-mixing sequences of multivariate random variables (see Lemma D.1). As Su and Ullah \cite{Su2013187} argued, the result is also valid for  $\alpha$-mixing triangular arrays of random variables (see Lemma A.2 in the supplement material to \cite{Su2013187}). 

\begin{lemma}\label{Su D.1}
Let $\{\xi_{n,i}:1\le i\le n,n\in\N\}$ be an $l$-dimensional strongly mixing triangular array of random variables with mixing coefficient $\alpha(\cdot)$. Let $F_{n,i_1\dots i_m}$ denote the distribution function of $(\xi_{n,i_1},\dots,\xi_{n,i_m})$. For some $m>1$ and integers $(i_1,\dots,i_m)$ such that $1\le i_1<\dots<i_m\le n$, let $g$ be a Borel measurable function such that
\[\int |g(x_1,\dots,x_m)|^{1+\delta}dF_{n,i_1\dots i_m}(x_1,\dots,x_m)\le M_n\]
and
\[\int |g(x_1,\dots,x_m)|^{1+\delta}dF_{n,i_1\dots i_j}(x_1,\dots,x_j)dF_{n,i_{j+1}\dots i_m}(x_{j+1},\dots,x_m)\le M_n,\]
for some $\delta>0$. Then, it holds that
\begin{align*}
&\left|\int g(x_1,\dots,x_m)dF_{n,i_1\dots i_m}(x_1,\dots,x_m)\right.\\
&\hspace{3cm}\left.-\int g(x_1,\dots,x_m)dF_{n,i_1\dots i_j}(x_1,\dots,x_j)dF_{n,i_{j+1}\dots i_m}(x_{j+1},\dots,x_m)\right|\\
&\le 4M_n^{\frac{1}{1+\delta}}\alpha(i_{j+1}-i_j)^{\frac{\delta}{1+\delta}}.
\end{align*}
\end{lemma}

Sun and Chiang gave a proof of their version Lemma 2.1 in \cite{Su19973}. The proof of the generalizing result Lemma \ref{Su D.1} works analogously and is therefore omitted.

\begin{proof}[Proof of Lemma \ref{moment inequality}]

The proof is closely related to the proof of Lemma 3.1 by Andrews \& Pollard \cite{Andrews1994549}, but uses the covariance inequality in Lemma \ref{Su D.1}. 

To begin with, it will be proven via induction that for all $Q\ge 2$ (not necessarily even) satisfying assumptions \textbf{(i)} and \textbf{(ii)}, there exists a constant $C'$ only depending on $Q, \gamma$ and the mixing coefficient, such that
\begin{equation}
\sum\limits_{\bm{i}\in \bm{I}^Q}\left|E\left[Z_{n,i_1}\dots Z_{n,i_Q}\right]\right|\le C' \left(n\tau^2+\dots +(n\tau^2)^{\lf\frac{Q}{2}\rf}\right), \ \forall \ n\in\N,\label{eq:moment inequality_01}
\end{equation}
where $\bm{I}^Q:=\{\bm{i}=(i_1,\dots,i_Q)\in\{1,\dots,n\}^Q:i_1\le \dots\le i_Q\}$.

Let for $Q=2$ the assumptions of Lemma \ref{moment inequality} hold for some $\gamma,\tau>0$. Let furthermore $\{\tilde{Z}_{n,t}:1\le t\le n,n\in\N\}$ be an independent copy of $\{Z_{n,t}:1\le t\le n,n\in\N\}$. Then applying \textbf{(ii)},
\begin{align*}
E\left[|Z_{n,i_1}\tilde{Z}_{n,i_2}|^{\frac{2+\gamma}{2}}\right]
&\le E\left[|Z_{n,i_1}|^{2\frac{2+\gamma}{2}}\right]^{\frac{1}{2}}E\left[|\tilde{Z}_{n,i_2}|^{2\frac{2+\gamma}{2}}\right]^{\frac{1}{2}}\\
&\le\tau^{2+\gamma}
\end{align*} 
and similarly
\begin{align*}
E\left[|Z_{n,i_1}Z_{n,i_2}|^{\frac{2+\gamma}{2}}\right]
&\le\tau^{2+\gamma}
\end{align*}
holds for all $i_1,i_2\in\{1,\dots,n\}$ with $i_1\neq i_2$ and $n\in\N$. Lemma \ref{Su D.1} can therefore be applied with $g(x_1,x_2):=x_1x_2$, $\delta:=\frac{\gamma}{2}$ and $M_n:=\tau^{2+\gamma}$ for all $n\in\N$. It implies that
\begin{align*}
\sum\limits_{i_1=1}^{n}\sum\limits_{\substack{i_2=1\\i_1< i_2}}^{n}\left|E\left[Z_{n,i_1}Z_{n,i_2}\right]\right|
&=\sum\limits_{i_1=1}^{n}\sum\limits_{\substack{i_2=1\\i_1< i_2}}^{n}\left|E\left[Z_{n,i_1}Z_{n,i_2}\right]-E\left[Z_{n,i_1}\right]E\left[Z_{n,i_2}\right]\right|\\
&\le \sum\limits_{i_1=1}^{n}\sum\limits_{\substack{i_2=1\\i_1< i_2}}^{n}\alpha(i_2-i_1)^{\frac{\gamma}{2+\gamma}}4\left(\tau^{2+\gamma}\right)^{\frac{2}{2+\gamma}}\\
&\le n\tau^2 4\sum\limits_{t=1}^{\infty}\alpha(t)^{\frac{\gamma}{2+\gamma}}\\
&= C'' n\tau^2,
\end{align*}
for $C'':=4\sum_{t=1}^{\infty}\alpha(t)^{\frac{\gamma}{2+\gamma}}<\infty$ by assumption \textbf{(i)}, a constant therefore only depending on $\gamma$ and the mixing coefficient. Furthermore, using this and
\begin{align*}
E\left[|Z_{n,i_1}|^2\right]\le E\left[|Z_{n,i_1}|^{2\frac{2+\gamma}{2}}\right]^{\frac{2}{2+\gamma}}\le \left(\tau^{2+\gamma}\right)^{\frac{2}{2+\gamma}}=\tau^2,
\end{align*}
for all $i_1\in\{1,\dots,n\}$ by assumption \textbf{(ii)}, leads to
\begin{align*}
\sum\limits_{i_1=1}^{n}\sum\limits_{\substack{i_2=1\\i_1\le i_2}}^{n}\left|E\left[Z_{n,i_1}Z_{n,i_2}\right]\right|
&=\sum\limits_{i_1=1}^{n}\left|E\left[Z_{n,i_1}^2\right]\right|+\sum\limits_{i_1=1}^{n}\sum\limits_{\substack{i_2=1\\i_1< i_2}}^{n}\left|E\left[Z_{n,i_1}Z_{n,i_2}\right]\right|\\
&\le n \tau^2+C''n\tau^2\\
&=C' n\tau^2,
\end{align*}
with $C':=(1+C'')$, which is the assertion of (\ref{eq:moment inequality_01}) for $Q=2$. For the inductive step, let now $Q>2$ be arbitrary, but fixed and let the assertion in (\ref{eq:moment inequality_01}) be true for all integers in $\{2,\dots,Q-1\}$. Then it is to show that (\ref{eq:moment inequality_01}) holds for $Q$ as well. Let therefore assumptions \textbf{(i)} and \textbf{(ii)} be satisfied for this arbitrary, but fixed $Q>2$ and for some $\gamma>0$ and $\tau>0$. Note that then the assumptions are satisfied for all integers in $\{2,\dots,Q-1\}$ as well.

The idea of the proof is the following. First, the smallest index where the gap between two succeeding indices is largest and positive (to exclude the case, where all indices are equal) is identified. The random variables after this time point will then be replaced by independent copies of themselves. For the new term, the induction hypothesis can be used as there will be less than $Q$ indices left. The remainder term can be bounded using Lemma \ref{Su D.1}. Following the notation of Andrews \& Pollard \cite{Andrews1994549} let for all $\bm{i}\in\bm{I}^Q$ 
\[G(\bm{i}):=\max\left\{(i_{j+1}-i_j): (i_{j+1}-i_j)>0, 1\le j \le Q-1\right\}\] 
and
\[m(\bm{i}):=\min\left\{j\in\{1,\dots,Q-1\}:(i_{j+1}-i_{j})=G(\bm{i})\right\}.\]

Then, it can be obtained that
\begin{align}
&\sum\limits_{\bm{i}\in \bm{I}^Q} \left|E\left[Z_{n,i_1}\cdots Z_{n,i_Q}\right]\right|\notag\\
&=\sum\limits_{\substack{\bm{i}\in \bm{I}^Q\\i_1=\dots=i_Q}}|E[Z_{n,i_1}\cdots Z_{n,i_Q}]|+\sum\limits_{m=1}^{Q-1}\sum\limits_{\substack{\bm{i}\in \bm{I}^Q\\ m(\bm{i})=m}}\left|E\left[Z_{n,i_1}\cdots Z_{n,i_Q}\right]\right|\notag\\
&\le\sum\limits_{i_1=1}^{n}\left|E\left[Z_{n,i_1}^Q\right]\right|\label{eq:induction_00}\\
&+\sum\limits_{m=1}^{Q-1}\sum\limits_{\substack{\bm{i}\in \bm{I}^Q\\ m(\bm{i})=m}} \left|E\left[Z_{n,i_1}\cdots Z_{n,i_Q}\right]-E\left[Z_{n,i_1}\cdots Z_{n,i_m}\right]E\left[Z_{n,i_{m+1}}\cdots Z_{n,i_Q}\right]\right|\label{eq:induction_01}\\
&+\sum\limits_{m=1}^{Q-1}\sum\limits_{\substack{\bm{i}\in \bm{I}^Q\\ m(\bm{i})=m}}\left|E\left[Z_{n,i_1}\cdots Z_{n,i_m}\right]E\left[Z_{n,i_{m+1}}\cdots Z_{n,i_Q}\right]\right|.\label{eq:induction_02}
\end{align}

Let first (\ref{eq:induction_00}) be considered. Using assumption \textbf{(ii)}, it holds that
\begin{align*}
\sum\limits_{i_1=1}^{n}\left|E\left[Z_{n,i_1}^Q\right]\right|
\le \sum\limits_{i_1=1}^{n} E\left[\left|Z_{n,i_1}\right|^{Q\frac{2+\gamma}{2}}\right]^{\frac{2}{2+\gamma}}
\le\sum\limits_{i_1=1}^{n} \left(\tau^{2+\gamma}\right)^{\frac{2}{2+\gamma}}
=n\tau^2.
\end{align*}

Let next (\ref{eq:induction_01}) be considered. Using H\"{o}lder's inequality and assumption \textbf{(ii)},
\begin{align*}
&E\left[|Z_{n,i_1}\cdots Z_{n,i_{k}} \tilde{Z}_{n,i_{k+1}} \cdots \tilde{Z}_{n,i_Q}|^{\frac{2+\gamma}{2}}\right]\\
&\le E\left[|Z_{n,i_1}|^{Q\frac{2+\gamma}{2}}\right]^{\frac{1}{Q}}\cdots E\left[|Z_{n,i_k}|^{Q\frac{2+\gamma}{2}}\right]^{\frac{1}{Q}}E\left[|\tilde{Z}_{n,i_{k+1}}|^{Q\frac{2+\gamma}{2}}\right]^{\frac{1}{Q}}\cdots E\left[|\tilde{Z}_{n,i_Q}|^{Q\frac{2+\gamma}{2}}\right]^{\frac{1}{Q}}\\
&\le \tau^{2+\gamma}
\end{align*}
and similarly 
\begin{align*}
E\left[|Z_{n,i_1}\cdots Z_{n,i_Q}|^{\frac{2+\gamma}{2}}\right]
&\le \tau^{2+\gamma}
\end{align*}
holds, for all $k\in\{1,\dots,Q-1\}$. Hence, applying Lemma \ref{Su D.1} with $g(x_1,\dots,x_Q):=x_1\cdots x_Q$, $\delta:=\frac{\gamma}{2}$ and $M_n:=\tau^{2+\gamma}$ for all $n\in\N$, implies for all $k\in\{1,\dots, Q-1\}$,
\begin{align*}
\left|E\left[Z_{n,i_1}\cdots Z_{n,i_Q}\right]-E\left[Z_{n,i_1}\cdots Z_{n,i_k}\right]E\left[Z_{n,i_{k+1}}\cdots Z_{n,i_Q}\right]\right|\le 4 \tau^2 \alpha(i_{k+1}-i_k)^{\frac{\gamma}{2+\gamma}}.
\end{align*}

Using this and distinguishing the indices furthermore by location of the gap $l\in\{1,\dots,n\}$ and size of the gap $g\in\{1,\dots,n\}$, (\ref{eq:induction_01}) can be bounded by
\begin{align}
4\tau^2\sum\limits_{m=1}^{Q-1}\sum\limits_{\substack{\bm{i}\in \bm{I}^Q\\ m(\bm{i})=m}} \alpha(i_{m+1}-i_m)^{\frac{\gamma}{2+\gamma}}
&=4\tau^2\sum\limits_{m=1}^{Q-1}\sum\limits_{l=1}^{n}\sum\limits_{g=1}^{n}\sum\limits_{\bm{i}\in\bm{I}_{m,g,l}^Q}\alpha(g)^{\frac{\gamma}{2+\gamma}}, \label{eq:induction_011}
\end{align}
where $\bm{I}^{Q}_{m,g,l}:=\{\bm{i}\in\bm{I}^{Q}:m(\bm{i})=m,G(\bm{i})=g,i_m=l\}$. A more detailed study of the set of indices $\bm{I}^{Q}_{m,g,l}$ will lead to a suitable bound for (\ref{eq:induction_011}). For fixed $m,l,g$, it will be investigated, how many elements the set $\bm{I}^{Q}_{m,g,l}$ contains at most. For all $\bm{i}=(i_1,\dots,i_Q)\in\bm{I}^{Q}_{m,g,l}$ the first $m-1$ indices $i_1,\dots,i_{m-1}$ satisfy
\begin{enumerate}
\item[$\bullet$] $1\le i_1\le \dots\le i_{m-1}\le i_m=l$,
\item[$\bullet$] $i_{j+1}-i_j< g$ for all $j=1,\dots,m-1$.
\end{enumerate}

With these restrictions, for fixed $i_2,\dots,i_m$, the sum over $i_1$ ranges over at most $g$ elements. For fixed $i_3,\dots,i_m$ the sum over $i_2$, again ranges over at most $g$ elements, with the above restrictions. Continuing in this way, there are at most $g^{m-1}$ choices for the first $m-1$ indices $i_1,\dots,i_{m-1}$. Because of $m(\bm{i})=m,i_m=l$ and $G(\bm{i})=g$, it holds that $i_{m+1}=l+g$ and therefore the last $Q-m-1$ indices $i_{m+2},\dots,i_Q$ satisfy the following restrictions
\begin{enumerate}
\item[$\bullet$] $l+g=i_{m+1}\le i_{m+2}\le \dots\le i_{Q}\le n$,
\item[$\bullet$] $i_{j+1}-i_j< g+1$ for all $j=m+1,\dots,Q$.
\end{enumerate}

Hence, following the same arguments as above, there are at most $(g+1)^{Q-m-1}$ choices for the last $Q-m-1$ indices $i_{m+2},\dots,i_Q$. Therefore (\ref{eq:induction_011}) can further be bounded by
\begin{align*}
&4\tau^2\sum\limits_{m=1}^{Q-1}\sum\limits_{l=1}^{n}\sum\limits_{g=1}^{n}g^{m-1}(g+1)^{Q-m-1}\alpha(g)^{\frac{\gamma}{2+\gamma}}\\
&\le 4\tau^2(Q-1)n\sum\limits_{g=1}^{\infty}(g+1)^{Q-2}\alpha(g)^{\frac{\gamma}{2+\gamma}}\\
&= C''n\tau^2,
\end{align*}
for $C'':=4(Q-1)\sum_{t=1}^{\infty}(t+1)^{Q-2}\alpha(t)^{\frac{\gamma}{2+\gamma}}<\infty$ by assumption \textbf{(i)}, a constant only depending on $Q$, $\gamma$ and the mixing coefficient. 

It is left to consider (\ref{eq:induction_02}). Introducing the following notation
\[B_n(i):=n\tau^2+\left(n\tau^2\right)^2+\dots+\left(n\tau^2\right)^{\lf \frac{i}{2}\rf} \ \forall \ i=1,\dots,Q-1, \ \forall \ n\in\N\]
and applying the induction hypotheses, it holds that
\begin{align*}
&\sum\limits_{m=1}^{Q-1}\sum\limits_{\substack{\bm{i}\in \bm{I}^Q\\ m(\bm{i})=m}}\left|E\left[Z_{n,i_1}\cdots Z_{n,i_m}\right]E\left[Z_{n,i_{m+1}}\cdots Z_{n,i_Q}\right]\right|\\
&\le\sum\limits_{m=1}^{Q-1}\sum\limits_{\bm{i}\in \bm{I}^m}\left|E\left[Z_{n,i_1}\cdots Z_{n,i_m}\right]\right|\sum\limits_{\bm{i}\in \bm{I}^{Q-m}}\left|E\left[Z_{n,i_{m+1}}\cdots Z_{n,i_Q}\right]\right|\\
&\le\sum\limits_{m=1}^{Q-1} C_{m}B_n(m)C_{Q-m}B_n(Q-m),
\end{align*}
for some constants $C_i$ only depending on $i$, $\gamma$ and the mixing coefficient for all $i=1,\dots,Q-1$. As $B_n(m)B_n(Q-m)$ is a polynomial in $n\tau^2$ of degree
\[\lf\frac{m}{2}\rf+\lf\frac{Q-m}{2}\rf\le\lf\frac{Q}{2}\rf,\]
there is a constant $C_{m,Q}$, such that $B_n(m)B_n(Q-m)\le C_{m,Q}B_n(Q)$. Hence, the above sum can be bounded by
\begin{align*}
\sum\limits_{m=1}^{Q-1}C_{m}C_{Q-m}C_{m,Q}B_n(Q)=C'''B_n(Q),
\end{align*}
for $C''':=\sum_{m=1}^{Q-1}C_{m}C_{Q-m}C_{m,Q}<\infty$. Putting the results for (\ref{eq:induction_00}), (\ref{eq:induction_01}) and (\ref{eq:induction_02}) together, it holds that
\begin{align*}
\sum\limits_{\bm{i}\in \bm{I}^Q}\left|E\left[Z_{n,i_1}\dots Z_{n,i_Q}\right]\right|
&\le n\tau^2+C''n\tau^2+C'''B_n(Q)\\
&\le C'\left(n\tau^2+\dots+(n\tau^2)^{\lf \frac{Q}{2} \rf }\right),
\end{align*}
for $C'=1+C''+C'''$ only depending on $Q$, $\gamma$ and the mixing coefficient, which completes the induction and therefore the proof of (\ref{eq:moment inequality_01}) for all $Q\ge 2$ satisfying the assumptions. 

Using
\begin{align*}
(n\tau^2)^i\le \max\left(1,(n\tau^2)^{\lf \frac{Q}{2} \rf}\right) \ \forall \ 1\le i\le \lf \frac{Q}{2} \rf
\end{align*}
it can be obtained that for $Q$ being even
\begin{align*}
E\left[\left|\frac{1}{\sqrt{n}}\sum\limits_{i=1}^{n}Z_{n,i}\right|^Q\right]^{\frac{1}{Q}}
&=n^{-\frac{1}{2}}E\left[\sum\limits_{i_1=1}^{n}\dots\sum\limits_{i_Q=1}^{n}Z_{n,i_1}\dots Z_{n,i_Q}\right]^{\frac{1}{Q}}\\
&\le n^{-\frac{1}{2}} (Q!)^{\frac{1}{Q}} \left(\sum\limits_{\bm{i}\in \bm{I}^Q}\left|E\left[Z_{n,i_1}\dots Z_{n,i_Q}\right]\right|\right)^{\frac{1}{Q}}\\
&\le n^{-\frac{1}{2}} (Q!)^{\frac{1}{Q}} C'^{\frac{1}{Q}} \left((n\tau^2)+\dots +(n\tau^2)^{\frac{Q}{2}}\right)^{\frac{1}{Q}}\\
&\le C \max\left(n^{-\frac{1}{2}},\tau\right),
\end{align*}
for $C'$ from inequality (\ref{eq:moment inequality_01}) and $C:=\left(C'Q!\frac{Q}{2}\right)^{\frac{1}{Q}}$ only depending on $Q,\gamma$ and the mixing coefficient, which proves the assertion of Lemma \ref{moment inequality}.

\end{proof}

Within the proof of Theorem \ref{conv}, let the following simplifying notation hold. Denote $G_n(\ph):=G_n(1,\ph)$ for measurable functions $\ph:\cX\to\R$.

\begin{proof}[Proof of Theorem \ref{conv}] 
The proof is closely related to the proof of Theorem 2.2 of Andrews \& Pollard \cite{Andrews1994549}. It will be shown that for all $\epsilon>0$, there exists a $\delta=\delta(\epsilon)>0$ and an $n_0=n_0(\epsilon)$, such that for all $n\ge n_0$,
\begin{align}
E^{\ast}\left[\sup\limits_{\{\ph,\psi\in\cF:d(\ph,\psi)<\delta\}}\left|G_n(\ph)-G_n(\psi)\right|^Q\right]^{\frac{1}{Q}}<\epsilon. \label{eq:AP}
\end{align}

Let therefore be $\epsilon>0$. Let for $k\in\N$, $\delta_k:=2^{-k}$, $\tau_k:=\delta_k^{\frac{2}{2+\gamma}}$ and $N_k:=N_{[~]}(\delta_k,\cF,\rho)$ and let $\cA_k$ be the approximating class and $\cB_k$ the bounding class from Definition \ref{brackets}, that are chosen such that, assumption (\ref{eq:bounding function}) in \textbf{(A2)} holds. In particular, it holds that for all $\ph\in\cF$, there exist an $a_k^*\in\cA_k$ and a $b_k^*\in \cB_k$, such that
\begin{align}
|\ph-a_k^*|\le b_k^*,\label{eq:partition1}
\end{align}
and for all $b\in\cB_k$
\begin{align}
\sup\limits_{n\in\N}\sup\limits_{1\le t\le n}E\left[|b(X_{n,t})|^{2}\right]^{\frac{1}{2}}\le\delta_k, \
\sup\limits_{n\in\N}\sup\limits_{1\le t\le n}E\left[|b(X_{n,t})|^{i\frac{2+\gamma}{2}}\right]^{\frac{1}{2}}\le\delta_k, \ \forall \ i=2,\dots,Q.\label{eq:partition2}
\end{align}

%
The proof splits into two steps. First, it will be shown that there exist an $m=m(\epsilon)$ and for each $\ph\in\cF$ a function $a_m^{(\ph)}\in\cA_m$ and an $n_1=n_1(\epsilon)$, such that, for all $n\ge n_1$,
\begin{align}
E^{\ast}\left[\sup\limits_{\ph\in\cF}\left|G_n(\ph)-G_n\left(a_m^{(\ph)}\right)\right|^Q\right]^{\frac{1}{Q}}<\frac{\epsilon}{8}. \label{eq:AP(3.4)}
\end{align}

Note that $a_m^{(\ph)}$ is not necessarily the corresponding approximating function, denoted by $a_m^*\in\cA_m$, from Definition \ref{brackets}, but rather results from a constructive argument, such that (\ref{eq:AP(3.4)}) holds. 

Secondly, for this fixed $m\in\N$, $\cF$ will be partitioned into $N_m$ many classes, each class containing all functions $\ph$ in $\cF$, that lead to the same $a_m^{(\ph)}\in\cA_m$, in step one. Within each class inequality (\ref{eq:AP(3.4)}) will be applied. By a right choice of functions the gap between two different classes can also be bridged suitably.

\textit{Step 1:} The proof of (\ref{eq:AP(3.4)}) is again divided into two parts. First a sequence $k(n)\to\infty$ and an $n_2=n_2(\epsilon)$ are chosen, such that for all $n\ge n_2$,
\begin{align}
E^{\ast}\left[\sup\limits_{\ph\in\cF}\left|G_n(\ph)-G_n\left(a_{k(n)}^*\right)\right|^Q\right]^{\frac{1}{Q}}<\frac{\epsilon}{16},\label{eq:AP(3.2)}
\end{align}
where for each $\ph\in\cF$ and $k(n)\in\N$, $a_{k(n)}^*$ is the corresponding approximating function in $\cA_{k(n)}$, as in (\ref{eq:partition1}) for $k=k(n)$. 

Secondly $m=m(\epsilon)$, and for each $\ph\in\cF$, $a_m^{(\ph)}\in\cA_m$ and $n_3=n_3(\epsilon)$ are chosen, such that for all $n\ge n_3$ with $k(n)>m$,
\begin{align}
E^{\ast}\left[\sup\limits_{\ph\in\cF}\left|G_n\left(a_{k(n)}^*\right)-G_n\left(a_{m}^{(\ph)}\right)\right|^Q\right]^{\frac{1}{Q}}<\frac{\epsilon}{16}.\label{eq:AP(3.3)}
\end{align}

Here, for each $\ph\in\cF$, $a_{k(n)}^*$ is the corresponding approximating function in $\cA_{k(n)}$ from (\ref{eq:partition1}), while $a_m^{(\ph)}\in\cA_m$ not necessarily is. It rather results from an iterative choice of functions $a_k\in\cA_k$ to $a_{k-1}\in\cA_{k-1}$ for $k=k(n),\dots,m+1$. The choice of $a_m:=a_m^{(\ph)}$ then depends on $\ph$ and $n$, as it is the last link in the chain, that starts with $a_{k(n)}:=a_{k(n)}^*$ (which depends on $\ph$ by Definition \ref{brackets}, despite the fact, that this is not reflected in the notation). Nevertheless, the choice of $m$ does only depend on $\epsilon$ eventually. Both (\ref{eq:AP(3.2)}) and (\ref{eq:AP(3.3)}) together imply (\ref{eq:AP(3.4)}) by choosing $n_1=\max(n_2,n_3)$.

\textit{Proof of (\ref{eq:AP(3.2)}):} Let $k(n)$ be the largest value of $k\in\N$, such that
\begin{align}
2^{-k\frac{2}{2+\gamma}}=\tau_k\ge n^{-\frac{1}{2}}. \label{eq:tau_k}
\end{align}

Note that then
\begin{align*}
\sqrt{n}\tau_{k(n)+1}^{\frac{2+\gamma}{2}}\le \sqrt{n}\left(n^{-\frac{1}{2}}\right)^{\frac{2+\gamma}{2}}=n^{-\frac{\gamma}{4}}\overset{n\to\infty}{\longrightarrow} 0
\end{align*}
holds. Using (\ref{eq:partition2}), it follows that
\begin{align*}
\sqrt{n}\sup\limits_{m\in\N}\sup\limits_{1\le t\le m}E\left[|b(X_{m,t})|^2\right]^{\frac{1}{2}}\le \sqrt{n}\delta_{k(n)}=\sqrt{n}2\delta_{k(n)+1}=\sqrt{n}2\tau_{k(n)+1}^{\frac{2+\gamma}{2}}=o(1), 
\end{align*}
for all $b\in\cB_{k(n)}$. Thus there exists an $n_2'=n_2'(\epsilon)$, such that
\begin{align}
2\sqrt{n}\sup\limits_{m\in\N}\sup\limits_{1\le t\le m}E\left[|b(X_{m,t})|^2\right]^{\frac{1}{2}}< \frac{\epsilon}{32},\label{eq:E-Wert}
\end{align}
for all $b\in\cB_{k(n)}$ and $n\ge n_2'$. Hence, for $\ph\in\cF$ and corresponding approximation function $a_{k(n)}^*\in\cA_{k(n)}$, applying (\ref{eq:partition1}), it holds that\\
\begin{align}
\left|G_n(\ph)-G_n\left(a_{k(n)}^*\right)\right|
&=\left|G_n\left(\ph-a_{k(n)}^*\right)\right|\notag\\
&\le \frac{1}{\sqrt{n}}\sum\limits_{i=1}^{n}\left(|\ph(X_{n,i})-a_{k(n)}^*(X_{n,i})|+E\left[|\ph(X_{n,i})-a_{k(n)}^*(X_{n,i})|\right]\right) \notag\\
&\le \frac{1}{\sqrt{n}}\sum\limits_{i=1}^{n}b_{k(n)}^*(X_{n,i})+\frac{1}{\sqrt{n}}\sum\limits_{i=1}^{n}E\left[b_{k(n)}^*(X_{n,i})\right]\notag\\
&=G_n\left(b_{k(n)}^*\right)+2\frac{1}{\sqrt{n}}\sum\limits_{i=1}^{n}E\left[b_{k(n)}^*(X_{n,i})\right]\notag\\
&\le G_n\left(b_{k(n)}^*\right)+2\sqrt{n}\sup\limits_{m\in\N}\sup\limits_{1\le t\le m}E\left[|b_{k(n)}^*(X_{m,t})|^2\right]^{\frac{1}{2}}\notag\\
&< G_n\left(b_{k(n)}^*\right)+\frac{\epsilon}{32},\label{eq:AP(3.2.1)}
\end{align}
for all $n\ge n_2'$, due to (\ref{eq:E-Wert}). Next, the moment inequality from Lemma \ref{moment inequality} will be applied to $G_n(b)$ for $b\in\cB_{k(n)}$. To do that, set for $b\in\cB_{k(n)}$ and $n\in\N$ fixed, 
\[Z_{m,t}:=b(X_{m,t})-E[b(X_{m,t})], \ \forall \ 1\le t\le m,m\in\N.\]

Then assumption \textbf{(ii)} of Lemma \ref{moment inequality} is satisfied, because for all $1\le t\le m,m\in\N$ and $i=2,\dots,Q$, it holds that $E[Z_{m,t}]=0$ and 
\begin{align*}
E\left[\left|Z_{m,t}\right|^{i\frac{2+\gamma}{2}}\right]
&=E\left[\left|b(X_{m,t})-E\left[b(X_{m,t})\right]\right|^{i\frac{2+\gamma}{2}}\right]\\
&\le 2^{i\frac{2+\gamma}{2}}E\left[\left|b(X_{m,t})\right|^{i\frac{2+\gamma}{2}}\right]\\
&\le 2^{Q\frac{2+\gamma}{2}}\sup\limits_{m\in\N}\sup\limits_{1\le t\le m}E\left[\left|b(X_{m,t})\right|^{i\frac{2+\gamma}{2}}\right]\\
&\le 2^{Q\frac{2+\gamma}{2}}\delta_{k(n)}^2\\
&=\left(2^{\frac{Q}{2}}\tau_{k(n)}\right)^{2+\gamma},
\end{align*}
due to (\ref{eq:partition2}). Assumption \textbf{(i)} of Lemma \ref{moment inequality} is also satisfied by \textbf{(A1)} for $\{X_{n,t}:1\le t\le n,n\in\N\}$ and inherited to $\{Z_{n,t}:1\le t\le n,n\in\N\}$. Applying Lemma \ref{moment inequality} to $Z_{n,t}$ and inequality (\ref{eq:tau_k}), it follows that for all $b\in\cB_{k(n)}$, there exists some constant $C$, only depending on $Q,\gamma$ and the mixing coefficient, such that
\begin{align}
E\left[\left|G_n(b)\right|^Q\right]^{\frac{1}{Q}}\le C \max\left(n^{-\frac{1}{2}},2^{\frac{Q}{2}}\tau_{k(n)}\right)=C' \tau_{k(n)}, \label{eq:Lemma auf b_k}
\end{align}
with $C':=C2^{\frac{Q}{2}}$. Finally using (\ref{eq:AP(3.2.1)}) and (\ref{eq:Lemma auf b_k}), it can be concluded that for all $n\ge n_2'$
\begin{align*}
E^{\ast}\left[\sup\limits_{\ph\in\cF}\left|G_n(\ph)-G_n\left(a_{k(n)}^*\right)\right|^Q\right]^{\frac{1}{Q}}
&\le E\left[\max\limits_{b\in\cB_{k(n)}}\left|G_n(b)\right|^Q\right]^{\frac{1}{Q}}+\frac{\ep}{32}\\
&\le N_{k(n)}^{\frac{1}{Q}}\max\limits_{b\in\cB_{k(n)}}E\left[\left|G_n(b)\right|^Q\right]^{\frac{1}{Q}}+\frac{\ep}{32}\\
&\le C' N_{k(n)}^{\frac{1}{Q}}\tau_{k(n)} +\frac{\ep}{32}\\
&= C' \left(N_{[~]}\left(\delta_{k(n)},\cF,\rho\right)\right)^{\frac{1}{Q}}\delta_{k(n)}^{\frac{2}{2+\gamma}}+\frac{\ep}{32}\\
&\le C' \int\limits_{0}^{\delta_{k(n)}}x^{-\frac{\gamma}{2+\gamma}}\left(N_{[~]}\left(x,\cF,\rho\right)\right)^{\frac{1}{Q}}dx+\frac{\ep}{32},
\end{align*}
where the last inequality uses the fact that $x\mapsto x^{-\frac{\gamma}{2+\gamma}}\left(N_{[~]}\left(x,\cF,\rho\right)\right)^{\frac{1}{Q}}$ is decreasing and the integral exists by assumption \textbf{(A2)}. As $\delta_{k(n)}\searrow 0$, there exists a $n_2''=n_2''(\epsilon)$, such that
\[C' \int\limits_{0}^{\delta_{k(n)}}x^{-\frac{\gamma}{2+\gamma}}\left(N_{[~]}\left(x,\cF,\rho\right)\right)^{\frac{1}{Q}}dx<\frac{\ep}{32},\]
for all $n\ge n_2''$. By choosing $n_2=\max(n_2',n_2'')$ the assertion in (\ref{eq:AP(3.2)}) follows. 

%
\textit{Proof of (\ref{eq:AP(3.3)}):} The aim is to choose an $m=m(\epsilon)$ fixed (dependent only on $\epsilon$ eventually) and for each $\ph\in\cF$ the corresponding approximating function $a_{k(n)}^*\in\cA_{k(n)}$. Then a chain from $a_{k(n)}:=a_{k(n)}^*\in\cA_{k(n)}$ to $a_m:=a_m^{(\ph)}\in\cA_m$ for all $k(n)>m$ is built. In what follows, the iterative choice of functions from one chain link $a_k$ to the next one $a_{k-1}$ will be illustrated. For an already chosen $a_k\in\cA_k$, choose $a_{k-1}\in\cA_{k-1}$, such that
\begin{align}
a_{k-1}\in\left\{a\in\cA_{k-1}:\max\limits_{2\le i\le Q}\sup\limits_{n\in\N}\sup\limits_{1\le t\le \N}E\left[\left|a_k(X_{n,t})-a(X_{n,t})\right|^{i\frac{2+\gamma}{2}}\right] \text{ is minimal }\right\}.\label{eq:minimizer}
\end{align}

Such an object exists as the considered term is bounded from below by zero. If there is more than one minimizer, then one of them is chosen randomly. By doing so,  while $\ph$ ranges over $\cF$, each difference $(a_k-a_{k-1})$ ranges over at most $N_k$ functions, because $a_k$ ranges over at most $|\cA_k|=N_k$ functions and $a_{k-1}$ is chosen according to the procedure above. Then it holds that
%
%
\begin{align}
E\left[\sup\limits_{\ph\in\cF}\left|G_n\left(a_k\right)-G_n\left(a_{k-1}\right)\right|^Q\right]^{\frac{1}{Q}}
&\le N_k^{\frac{1}{Q}}\sup\limits_{\ph\in\cF}E\left[\left|G_n\left(a_k\right)-G_n\left(a_{k-1}\right)\right|^Q\right]^{\frac{1}{Q}}\notag\\
&=N_k^{\frac{1}{Q}}\sup\limits_{\ph\in\cF}E\left[\left|G_n\left(a_k-a_{k-1}\right)\right|^Q\right]^{\frac{1}{Q}}.\label{eq:AP(3.3.1)}
\end{align}

Notice again that unlike the supremum suggests, the possible difference ranges over finitely many functions and therefore the inequality, used in (\ref{eq:AP(3.3.1)}), is valid and the outer expectation simplifies to the usual expectation. For the expected value of the last term, the moment inequality from Lemma \ref{moment inequality} will be used again. Defining 
\[Z_{n,t}:=a_k(X_{n,t})-a_{k-1}(X_{n,t})-E[a_k(X_{n,t})-a_{k-1}(X_{n,t})], \ \forall \ 1\le t\le n,n\in\N,\] 
it can be obtained that $E[Z_{n,t}]=0$ for all $1\le t\le n,n\in\N$. Furthermore, by assumption \textbf{(A2)}, for $a_k\in\cA_k\subset\cF$, there exist an $\tilde{a}_{k-1}^*\in\cA_{k-1}$ and a $\tilde{b}_{k-1}^*\in\cB_{k-1}$, such that
\[\left|a_k-\tilde{a}_{k-1}^*\right|\le \tilde{b}_{k-1}^* \text{ and } \sup\limits_{n\in\N}\sup\limits_{1\le t\le n}E\left[\left|b(X_{n,t})\right|^{i\frac{2+\gamma}{2}}\right]^{\frac{1}{2}}\le \delta_{k-1}, \ \forall \ i=2,\dots,Q, \ \forall \ b\in\cB_{k-1}.\]

Using (\ref{eq:minimizer}), it thus holds that
\begin{align*}
\max\limits_{2\le i\le Q}E\left[|Z_{n,t}|^{i\frac{2+\gamma}{2}}\right]
&=\max\limits_{2\le i\le Q}E\left[\left|a_k(X_{n,t})-a_{k-1}(X_{n,t})-E[a_k(X_{n,t})-a_{k-1}(X_{n,t})]\right|^{i\frac{2+\gamma}{2}}\right]\\
&\le 2^{Q\frac{2+\gamma}{2}}\max\limits_{2\le i\le Q}\sup\limits_{n\in\N}\sup\limits_{1\le t\le n}E\left[\left|a_k(X_{n,t})-a_{k-1}(X_{n,t})\right|^{i\frac{2+\gamma}{2}}\right]\\
&\le 2^{Q\frac{2+\gamma}{2}}\max\limits_{2\le i\le Q}\sup\limits_{n\in\N}\sup\limits_{1\le t\le n}E\left[\left|a_k(X_{n,t})-\tilde{a}_{k-1}^*(X_{n,t})\right|^{i\frac{2+\gamma}{2}}\right]\\
&\le 2^{Q\frac{2+\gamma}{2}}\max\limits_{2\le i\le Q}\sup\limits_{n\in\N}\sup\limits_{1\le t\le n}E\left[\left|\tilde{b}_{k-1}^*(X_{n,t})\right|^{i\frac{2+\gamma}{2}}\right]\\
&\le 2^{Q\frac{2+\gamma}{2}}\delta_{k-1}^2=\left(2^{\frac{Q}{2}}\tau_{k-1}\right)^{2+\gamma},
\end{align*}
and thus assumption \textbf{(ii)} of Lemma \ref{moment inequality} is satisfied. Assumption \textbf{(i)} is satisfied by \textbf{(A1)} for $\{X_{n,t}:1\le t\le n,n\in\N\}$ and is inherited to $\{Z_{n,t}:1\le t\le n,n\in\N\}$. Then applying Lemma \ref{moment inequality} and using $\tau_k\ge n^{-\frac{1}{2}}$ for all $1\le k\le k(n)$ by (\ref{eq:tau_k}), yield to 
\begin{align}
E\left[\left|G_n\left(a_k-a_{k-1}\right)\right|^Q\right]^{\frac{1}{Q}}\le C\tau_{k-1}, \label{eq:AP(3.3.2)}
\end{align}
for some constant $C$ only depending on $\gamma$, $Q$ and the mixing coefficient. Now all tools to build the bridge between $a_{k(n)}:=a_{k(n)}^*$ and $a_m:=a_m^{(\ph)}$ are obtained. Using (\ref{eq:AP(3.3.1)}) and (\ref{eq:AP(3.3.2)}), it holds that
\begin{align*}
E\left[\sup\limits_{\ph\in\cF}\left|G_n\left(a_{k(n)}^*\right)-G_n\left(a_{m}^{(\ph)}\right)\right|^Q\right]^{\frac{1}{Q}}
&=E\left[\sup\limits_{\ph\in\cF}\left|\sum\limits_{k=m+1}^{k(n)}\left(G_n\left(a_{k}\right)-G_n\left(a_{k-1}\right)\right)\right|^Q\right]^{\frac{1}{Q}}\\
&\le\sum\limits_{k=m+1}^{k(n)}E\left[\sup\limits_{\ph\in\cF}\left|G_n\left(a_{k}\right)-G_n\left(a_{k-1}\right)\right|^Q\right]^{\frac{1}{Q}}\\
&\le \sum\limits_{k=m+1}^{k(n)}N_k^{\frac{1}{Q}}\sup\limits_{\ph\in\cF}E\left[\left|G_n\left(a_k-a_{k-1}\right)\right|^Q\right]^{\frac{1}{Q}}\\
&\le C \sum\limits_{k=m+1}^{k(n)}N_k^{\frac{1}{Q}}\tau_{k-1}\\
&\le 2^{\frac{2}{2+\gamma}}C \sum\limits_{k=m+1}^{\infty} \left(N_{[~]}\left(\delta_k,\cF,\rho\right)\right)^{\frac{1}{Q}}\delta_k^{\frac{2}{2+\gamma}} \\
&=2^{\frac{2}{2+\gamma}+1}C \sum\limits_{k=m+1}^{\infty}\delta_k^{-\frac{\gamma}{2+\gamma}} \left(N_{[~]}\left(\delta_k,\cF,\rho\right)\right)^{\frac{1}{Q}}(\delta_k-\delta_{k-1})\\
&\le 2^{\frac{2}{2+\gamma}}C \int\limits_{0}^{\delta_m} x^{-\frac{\gamma}{2+\gamma}}\left(N_{[~]}\left(x,\cF,\rho\right)\right)^{\frac{1}{Q}}dx,
\end{align*}
for all $k(n)>m$. The last equality holds because $\delta_k-\delta_{k-1}=2^{-1}\delta_k$. The last inequality again holds as $x\mapsto x^{-\frac{\gamma}{2+\gamma}}\left(N_{[~]}\left(x,\cF,\rho\right)\right)^{\frac{1}{Q}}$ is decreasing and the integral exists due to assumption \textbf{(A2)}. Furthermore, $\delta_m\searrow 0$ for $m\to\infty$. Hence, for a given $\epsilon>0$, $m=m(\epsilon)$ and $n_3=n_3(\epsilon)$ can be chosen large enough, such that 
\begin{align*}
E\left[\sup\limits_{\ph\in\cF}\left|G_n\left(a_{k(n)}^*\right)-G_n\left(a_{m}^{(\ph)}\right)\right|^Q\right]^{\frac{1}{Q}}
<\frac{\epsilon}{16}
\end{align*}
for all $n\ge n_3$ with $k(n)>m$, which proves inequality (\ref{eq:AP(3.3)}).

%
\textit{Step 2:} In the second and last step of the proof, the comparison of infinitely many functions in $\cF$ will be reduced to finitely many functions, making use of inequality (\ref{eq:AP(3.4)}). To do that, let $m\in\N$ be the integer fixed in step one and refer with $a_m^{(\ph)}$ to the element in $\cA_m$, that is chosen dependent on $\ph\in\cF$, according to the procedure in step one. Let the following relation on $\cF$ (dependent on $m$) be introduced
\[\ph\sim_{m} \psi \Leftrightarrow a_m^{(\ph)}=a_m^{(\psi)}.\]

This relation is obviously an equivalence relation and, as $|\cA_m|=N_m$, partitions $\cF$ into $N_m$ many equivalence classes, denoted by 
\[\cE^{(m)}[1],\dots,\cE^{(m)}[N_m].\]

Each class thus contains all $\ph$ in $\cF$, that have the same $a_m^{(\ph)}$ in $\cA_m$, that has been chosen in step one. Within one equivalence class, inequality (\ref{eq:AP(3.4)}) can be applied twice, leading to
\begin{align}
&E^{\ast}\left[\sup\limits_{\{\ph,\psi\in\cF| \ph\sim_{m} \psi\}}\left|G_n(\ph)-G_n(\psi)\right|^Q\right]^{\frac{1}{Q}}\notag\\
&=E^{\ast}\left[\sup\limits_{\{\ph,\psi\in\cF|  \ph\sim_{m} \psi\}}\left|\left(G_n(\ph)-G_n(a_m^{(\ph)})\right)-\left(G_n(\psi)-G_n(a_m^{(\psi)})\right)\right|^Q\right]^{\frac{1}{Q}}\notag\\
&\le 2E^{\ast}\left[\sup\limits_{\ph\in\cF}\left|G_n(\ph)-G_n(a_m^{(\ph)})\right|^Q\right]^{\frac{1}{Q}}\notag\\
&< \frac{\epsilon}{4},\label{eq:AP(3.6)}
\end{align}
for all $n\ge n_1$. To bridge the gap between the $N_m$ classes, let
\[d(\cE^{(m)}[k],\cE^{(m)}[j]):=\inf\big\{d(\ph,\psi):\ph\in\cE^{(m)}[k],\psi\in\cE^{(m)}[j]\big\}\]
define a distance between two classes $\cE^{(m)}[k]$ and $\cE^{(m)}[j]$ for $k,j\in\{1,\dots,N_m\}$. For some fixed $\delta>0$, that will be specified later, and fixed $k,j\in\{1,\dots,N_m\}$, choose functions $\ph_{kj}'\in\cE^{(m)}[k]$ and $\psi_{jk}'\in\cE^{(m)}[j]$, such that 
\[d(\ph_{kj}',\psi_{jk}')<d(\cE^{(m)}[k],\cE^{(m)}[j])+\delta.\] 

Note that for $\ph\in\cE^{(m)}[k]$ and $\psi\in\cE^{(m)}[j]$ with $d(\ph,\psi)<\delta$, it holds that $d(\ph_{kj}',\psi_{jk}')<2\delta$ for all $k,j\in\{1,\dots,N_m\}$. Then applying (\ref{eq:AP(3.6)}) for all $n\ge n_1$, it can be obtained that
%
%
%
\begin{align*}
&E^{\ast}\left[\sup\limits_{\{\ph,\psi\in\cF:d(\ph,\psi)<\delta\}}|G_n(\ph)-G_n(\psi)|^Q\right]^{\frac{1}{Q}}\\
&=E^{\ast}\left[\max\limits_{\substack{1\le k\le N_m\\1\le j\le N_m}}\sup\limits_{\substack{\{\ph\in\cE^{(m)}[k],\psi\in\cE^{(m)}[j]:\\d(\ph,\psi)<\delta\}}}|G_n(\ph)-G_n(\psi)|^Q\right]^{\frac{1}{Q}}\\
&=E^{\ast}\left[\max\limits_{\substack{1\le k\le N_m\\1\le j\le N_m}}\sup\limits_{\substack{\{\ph\in\cE^{(m)}[k],\psi\in\cE^{(m)}[j]:\\d(\ph,\psi)<\delta\}}}|G_n(\ph)-G_n(\psi)\pm G_n(\ph'_{kj})\pm G_n(\psi'_{jk})|^Q\right]^{\frac{1}{Q}}\\
&\le 2 E^{\ast}\left[\max\limits_{1\le k\le N_m}\sup\limits_{\{\ph,\ph'\in\cE^{(m)}[k]\}}|G_n(\ph)-G_n(\ph')|^{Q}\right]^{\frac{1}{Q}}
+E\left[\max\limits_{\substack{1\le k\le N_m\\1\le j\le N_m}}|G_n(\ph'_{kj})- G_n(\psi'_{jk})|^Q\right]^{\frac{1}{Q}}\\
&<\frac{\epsilon}{2}+E\left[\max\limits_{\substack{1\le k\le N_m\\1\le j\le N_m}}|G_n(\ph'_{kj})- G_n(\psi'_{jk})|^Q\right]^{\frac{1}{Q}}\\
&\le \frac{\epsilon}{2}+N_m^{\frac{2}{Q}}\max\limits_{\substack{1\le k\le N_m\\1\le j\le N_m}}E\left[|G_n(\ph'_{kj}-\psi'_{jk})|^Q\right]^{\frac{1}{Q}},
\end{align*}
where $d(\ph'_{kj},\psi'_{jk})<2\delta$ holds for all $k,j\in\{1,\dots,N_m\}$. To find a bound on $E\left[|G_n(\ph'_{kj}-\psi'_{jk})|^Q\right]^{\frac{1}{Q}}$, the moment inequality of Lemma \ref{moment inequality} will be used. Let therefore $k,j\in\{1,\dots,N_m\}$ be fixed and let $\ph_{kj}'\in\cE^{(m)}[k]$ and $\psi_{jk}'\in\cE^{(m)}[j]$ with 
\begin{align}
d(\ph_{kj}',\psi_{jk}')=\sup\limits_{n\in \N}\sup\limits_{1\le t\le\N}E\left[\left|\ph_{kj}'(X_{n,t})-\psi_{jk}'(X_{n,t})\right|^{Q\frac{2+\gamma}{2}}\right]^{\frac{1}{Q}\frac{2}{2+\gamma}}< 2\delta \label{eq:d(ph',psi')}
\end{align}
and set $Z_{n,t}:=\ph_{kj}'(X_{n,t})-\psi_{jk}'(X_{n,t})-E[\ph_{kj}'(X_{n,t})-\psi_{jk}'(X_{n,t})]$ for all $1\le t\le n$ and $n\in\N$. Then, assumption \textbf{(ii)} of Lemma \ref{moment inequality} is satisfied as for all $1\le t\le n,n\in\N$ and $i=2,\dots,Q$ it holds that $E[Z_{n,t}]=0$ and
\begin{align*}
&E\left[\left|\ph_{kj}'(X_{n,t})-\psi_{jk}'(X_{n,t})-E[\ph_{kj}'(X_{n,t})-\psi_{jk}'(X_{n,t})]\right|^{i\frac{2+\gamma}{2}}\right]\\
&\le 2^{Q\frac{2+\gamma}{2}}E\left[|\ph_{kj}'(X_{n,t})-\psi_{jk}'(X_{n,t})|^{Q\frac{2+\gamma}{2}}\right]^{\frac{i}{Q}}\\
&\le 2^{Q\frac{2+\gamma}{2}}d(\ph_{kj}',\psi_{jk}')^{i\frac{2+\gamma}{2}}\\
&\le 2^{Q\frac{2+\gamma}{2}}(2\delta)^{i\frac{2+\gamma}{2}}\\
&\le 2^{Q\frac{2+\gamma}{2}}(2\delta)^{2+\gamma}, \ \text{ for } \delta\le\frac{1}{2}\\
&=\left(2^{\frac{Q}{2}+1}\delta\right)^{2+\gamma},
\end{align*}
due to (\ref{eq:d(ph',psi')}). Assumption \textbf{(i)} is satisfied by \textbf{(A1)} for $\{X_{n,t}:1\le t\le n,n\in\N\}$ and is inherited to $\{Z_{n,t}:1\le t\le n,n\in\N\}$. Applying Lemma \ref{moment inequality} yields to
\begin{align*}
E\left[|G_n(\ph_{kj}'-\psi_{jk}')|^Q\right]^{\frac{1}{Q}}\le C\max\left(n^{-\frac{1}{2}},2^{\frac{Q}{2}+1}\delta\right),
\end{align*}
for some constant $C$ only depending on $\gamma$, $Q$ and the mixing coefficient and for $\delta\le\frac{1}{2}$. Therefore, it holds that
\begin{align*}
E^{\ast}\left[\sup\limits_{\{\ph,\psi\in\cF:d(\ph,\psi)<\delta\}}\left|G_n(\ph)-G_n(\psi)\right|^Q\right]^{\frac{1}{Q}}
&< \frac{\epsilon}{2}+N_m^{\frac{2}{Q}}\max\limits_{\substack{1\le k\le N_m\\1\le j\le N_m}}E\left[|G_n(\ph_{kj}'-\psi_{jk}')|^Q\right]^{\frac{1}{Q}}\\
&<\frac{\epsilon}{2}+N_m^{\frac{2}{Q}}C\max\left(n^{-\frac{1}{2}},2^{\frac{Q}{2}+1}\delta\right)
\end{align*}
for all $n\ge n_1$ and for $\delta\le\frac{1}{2}$. Choose $\delta=\delta(\epsilon)$ small enough, such that
\[N_m^{\frac{2}{Q}}C2^{\frac{Q}{2}+1}\delta<\frac{\epsilon}{2} \text{ and } \delta\le\frac{1}{2}\]
and for this fixed $\delta$, let $n_4=n_4(\epsilon)$, such that
\[\max\left(n^{-\frac{1}{2}},2^{\frac{Q}{2}+1}\delta\right)=2^{\frac{Q}{2}+1}\delta,\]
for all $n\ge n_4$. By finally choosing $n_0:=\max(n_1,n_4)$, the assertion in (\ref{eq:AP}) is proven.

\end{proof}

\begin{proof}[Proof of Corollary \ref{cor1}]
This is a direct consequence of Theorem \ref{conv} and Markov's inequality.

\end{proof}

\begin{proof}[Proof of Corollary \ref{cor2}]
To prove Corollary \ref{cor2}, Theorem 4.10 by Volgushev \& Shao \cite{Volgushev2014390} will be applied. Note that it particularly requires a strictly stationary sequence of random variables. Applying Theorem \ref{conv}, it follows that there exists a semi-metric $d$ on $\cF$, such that $(\cF,d)$ is totally bounded and there exists a $Q>2$, such that
\[\lim\limits_{\delta\searrow 0}\limsup\limits_{n\to\infty}E^{\ast}\left[\sup\limits_{\{\ph,\psi\in\cF:d(\ph,\psi)<\delta\}}\left|G_n(1,\ph-\psi)\right|^Q\right]=0,\]
which is condition (9) of Theorem 4.10 of Volgushev \& Shao \cite{Volgushev2014390}. Furthermore, condition (10) of Theorem 4.10 in \cite{Volgushev2014390}, namely
\[\sup\limits_{n\in\N}\sup\limits_{\ph\in\cF}E\left[|G_n(1,\ph)|^Q\right]<\infty\]
is also satisfied. To see this, define $Z_t:=\ph(X_t)-E[\ph(X_t)]$ for all $t\in\Z$. Applying \textbf{(A3)}, it then holds that $E[Z_1]=0$ and for all $i=2,\dots,Q$
\begin{align*}
E\left[|Z_1|^{i\frac{2+\gamma}{2}}\right]
&=E\left[\left|\ph(X_1)-\int \ph dP\right|^{i\frac{2+\gamma}{2}}\right]\\
&\le 2^{Q\frac{2+\gamma}{2}}E\left[|\ph(X_1)|^{i\frac{2+\gamma}{2}}\right]\\
&\le 2^{Q\frac{2+\gamma}{2}}E\left[|\ph(X_1)|^{Q\frac{2+\gamma}{2}}\right]^{\frac{i}{Q}}\\
&\le 2^{Q\frac{2+\gamma}{2}}\max\left(M^{\frac{2}{Q}},M\right)\\
&=:\tau^{2+\gamma},
\end{align*}
for $M<\infty$ from assumption \textbf{(A3)}. Applying Lemma \ref{moment inequality}, it holds that for all $n\in\N$ and $\ph\in\cF$
\begin{align}
E\left[|G_n(1,\ph)|^Q\right]^{\frac{1}{Q}}\le C \max \left(n^{-\frac{1}{2}},\tau\right), \label{eq:cor2:(10)}
\end{align}
for some constant $C$, only depending on $Q$, $\gamma$ and the mixing coefficient. As the inequality in \textbf{(A3)} holds uniformly in $\ph\in\cF$, the constant $\tau<\infty$ does not depend on $\ph$. Therefore (\ref{eq:cor2:(10)}) implies condition (10) of Theorem 4.10 in \cite{Volgushev2014390}.

By assumption \textbf{(A3)} the function class $\cF$ possesses an envelope function with finite $Q$-th moment. Furthermore, all finite dimensional distributions converge by assumption. Applying Theorem 4.10 of Volgushev \& Shao \cite{Volgushev2014390}, the assertion of Corollary \ref{cor2} follows.

\end{proof}

\section*{Acknowledgements} 

I am deeply thankful to my doctoral advisor, Prof.~Dr.~Natalie Neumeyer, for a careful reading of this manuscript and helpful suggestions. I also want to thank Prof.~Dr.~Stanislav Volgushev for useful literature suggestions.

\bibliography{mybibfile}

\begin{thebibliography}{10}

\bibitem{Andrews1994549}
D.~W.~K. Andrews and D.~Pollard.
\newblock {An Introduction to Functional Central Limit Theorems for Dependent
  Stochastic Processes}.
\newblock {\em Int. Stat. Rev.}, 62:119--132, 1994.

\bibitem{Bradley1985165}
R.~C. Bradley.
\newblock {Basic properties of strong mixing conditions}.
\newblock In E.~Eberlein and M.~S. Taqqu, editors, {\em Dependence in
  Probability and Statistics}, pages 165--192. Birkh{\"{a}}user, Boston, 1985.

\bibitem{Burke2013261}
M.~D. Burke and G.~Bewa.
\newblock {Change-Point Detection for General Nonparametric Regression Models}.
\newblock {\em Open Journal of Statistics}, 3:261--267, 2013.

\bibitem{Dedecker2002137}
J.~Dedecker and S.~Louhichi.
\newblock {Maximal Inequalities and Empirical Central Limit Theorems}.
\newblock In H.~Dehling, T.~Mikosch, and M.~S{\o}rensen, editors, {\em
  {Empirical Process Techniques for Dependent Data}}, pages 137--159.
  Birkh{\"{a}}user, Boston, 2002.

\bibitem{Dehling201487}
H.~Dehling, O.~Durieu, and M.~Tusche.
\newblock {A sequential empirical CLT for multiple mixing processes with
  application to $\mathcal{B}$-geometrically ergodic Markov chains}.
\newblock {\em Electron. J. Probab.}, 19:1--26, 2014.

\bibitem{Dehling20141372}
H.~Dehling, O.~Durieu, and M.~Tusche.
\newblock {Approximating class approach for empirical processes of dependent
  sequences indexed by functions}.
\newblock {\em Bernoulli}, 20:1372--1403, 2014.

\bibitem{Doukhan1995}
P.~Doukhan.
\newblock {\em {Mixing: Properties and Examples}}.
\newblock Springer, New York, 2004.

\bibitem{Doukhan1996393}
P.~Doukhan, P.~Massart, and E.~Rio.
\newblock {Invariance principles for absolutely regular empirical processes}.
\newblock {\em Ann. Inst. H. Poincar\'{e}}, 31:393--427, 1995.

\bibitem{Hagemann2014188}
A.~Hagemann.
\newblock {Stochastic equicontinuity in nonlinear times series models}.
\newblock {\em Econom. J.}, 17:188--196, 2014.

\bibitem{Hansen1996347}
B.~E. Hansen.
\newblock {Stochastic equicontinuity for unbounded dependent heterogeneous
  arrays}.
\newblock {\em Econom. Theory}, 12:347--359, 1996.

\bibitem{Hariz2005339}
S.~B. Hariz.
\newblock {Uniform CLT for empirical process}.
\newblock {\em Stochastic Process. Appl.}, 115:339--358, 2005.

\bibitem{Massart1987}
P.~Massart.
\newblock {Invariance principles for empirical processes: the weakly dependent
  case}.
\newblock In {\em Chapter 1B of Ph.D. dissertation}, pages 58--102. University
  of Paris-South, Orsay, 1987.

\bibitem{Mohr2019}
M.~Mohr and N.~Neumeyer.
\newblock {Consistent nonparametric change point detection combining CUSUM and
  marked empirical processes}.
\newblock preprint on arXiv at \url{https://arxiv.org/abs/1901.08491}, January
  2019.

\bibitem{Ossiander1987897}
M.~Ossiander.
\newblock {A central limit theorem under metric entropy with $L_2$-bracketing}.
\newblock {\em Ann. Probab.}, 15:897--919, 1987.

\bibitem{Politis1999}
D.~N. Politis, J.~P. Romano, and M.~Wolf.
\newblock {\em {Subsampling}}.
\newblock Springer, New York, 1999.

\bibitem{Su2013187}
L.~Su and A.~Ullah.
\newblock {A nonparametric Goodness-of-fit-based test for conditional
  heteroscedasticity}.
\newblock {\em Econom. Theory}, 29:187--212, 2013.

\bibitem{Su2008347}
L.~Su and Z.~Xiao.
\newblock {Testing structural change in time-series nonparametric regression
  models}.
\newblock {\em Stat. Interface}, 1:347--366, 2008.

\bibitem{Su19973}
S.~Su and C.-Y. Chiang.
\newblock {Limiting behavior of the perturbed empirical distribution functions
  evaluated at U-statistics for strongly mixing sequences of random variables}.
\newblock {\em J. Appl. Math. Stoch. Anal.}, 10:3--20, 1997.

\bibitem{Tong1983}
H.~Tong.
\newblock {\em {Threshold Models in Non-linear Times Series Analysis}}.
\newblock Springer, New York, 1983.

\bibitem{vanderVaart1996}
A.~W. van~der Vaart and J.~A. Wellner.
\newblock {\em {Weak convergence and empirical processes}}.
\newblock Springer, New York, 1996.

\bibitem{Volgushev2014390}
S.~Volgushev and X.~Shao.
\newblock {A general approach to the joint asymptotic analysis of statistics
  from sub-samples}.
\newblock {\em Electron. J. Stat.}, 8:390--431, 2014.

\end{thebibliography}

\end{document}